\documentclass[11pt]{amsart}
\RequirePackage[l2tabu, orthodox]{nag}
\usepackage[T1]{fontenc}
\usepackage[utf8]{inputenc}
\usepackage{lmodern}
\usepackage[english]{babel}
\usepackage{tikz}
\usetikzlibrary{calc}
\usepackage{amsmath}
\usepackage{amsfonts}
\usepackage{amssymb}
\usepackage{amsthm}
\usepackage{enumerate}
\usepackage{multicol}
\usepackage{appendix}
\usepackage[most]{tcolorbox}

\newtcolorbox{mytextbox}[1][]{%
	sharp corners,
	enhanced,
	colback=white,
	height=20.5cm,
	attach title to upper,
	#1
}

\usepackage[margin=2.5cm]{geometry}
\newtheorem{theorem}{Theorem}[section]
\newtheorem{lemma}[theorem]{Lemma}

\newtheorem*{claim*}{Claim}

\newtheorem{conjecture}[theorem]{Conjecture}
\newtheorem{theoremA}{Theorem}

\newtheorem*{theorem*}{Theorem}
\newtheorem*{question*}{Question}
\newtheorem{corollary}[theorem]{Corollary}
\newtheorem*{corollary*}{Corollary}

\newtheorem{question}[theorem]{Question}

\theoremstyle{definition}
\newtheorem{definition}[theorem]{Definition}
\newtheorem*{definition*}{Definition}

\theoremstyle{remark}

\newtheorem{remark}[theorem]{Remark}
\newtheorem*{acknowledgment*}{Acknowledgment}

\numberwithin{equation}{section}
\numberwithin{figure}{section}

\def\NN{{\mathbb N}}

\def\RR{{\mathbb R}}

\def\ZZ{{\mathbb Z}}

\def\idid{{\mathrm{id}}}
\def\aff{{\mathrm{Aff}}}
\def\homeo{{\mathrm{Homeo}}}
\def\diffeo{{\mathrm{Diff}}}

\def\isom{{\mathrm{Isom}}}
\def\fix{{\mathrm{Fix}}}

\newcommand{\eqnum}{\refstepcounter{equation}\textup{\tagform@{\theequation}}}

\usepackage{indentfirst}

\usepackage[colorlinks=true,linkcolor=blue,citecolor=magenta]{hyperref}

\title[Groups acting on the line with at most $2$ fixed points]{Groups acting on the line with at most $2$ fixed points: an extension of Solodov's theorem}

\author{Jo\~ao Carnevale}
\date{\today}

\begin{document}
 
\maketitle

\begin{abstract}
	A classical result by Solodov states that if a group acts on the line such that any non-trivial element has at most one fixed point, then the action is either abelian or semi-conjugate to an affine action. We show that the same holds if we relax the assumption, requiring that any non-trivial element has at most 2 fixed points.
	
	\smallskip
	
\end{abstract}

\section{Introduction}

	Let $X$ be a topological space, and $G\le \homeo(X)$ a subgroup of homeomorphisms of $X$. Given $N\in\NN$, we say that $G$ \emph{has at most $N$ fixed points} if every non-trivial element of $G$ has at most $N$ fixed points.
	When $N=0$, we simply say that the action of $G$ on $X$ is \emph{free}.
A natural problem is to characterize actions with at most $N$ fixed points: 
\begin{question}\label{q.general}
	Given $N\in\NN$, which subgroups of $\homeo(X)$ have at most $N$ fixed points?
\end{question}

Here we will address this problem for the case when $X=\RR$ and $N=2$, extending classical works by H\"older (discussing the case $N=0$), and Solodov (case $N=1$).
These classical results essentially state that the only examples of groups with such property are the classical ones, namely the group  $\isom_+(\RR)$ of translations ($N=0$) and the group  $\mathrm{Aff}(\mathbb{R})$ of affine transformations of $\RR$ ($N=1$).  To state the results, given a subgroup $G\le \homeo(\RR)$, we denote by $G_+=G\cap \homeo_+(\RR)$ its normal subgroup (of index at most 2) of elements preserving the orientation of $\RR$.

\begin{theorem}[H\"older]\label{t.holderR} Every subgroup $G\leq \homeo(\mathbb{R})$ acting freely on $\mathbb{R}$  is abelian and moreover its action is semi-conjugate to an action by translations.
\end{theorem}
Actually, H\"older's original paper \cite{Holder} was only considering abelian groups, and the general statement was later extended by Conrad to any group in \cite[Section 3.8]{Conrad} (see also \cite[Theorem 6.10]{ghys-circle} for a more recent exposition).

\begin{theorem}[Solodov]\label{t.solodov} Let $G\leq \homeo(\mathbb{R})$ be a subgroup with at most $1$ fixed point.  Then 
	\begin{itemize}
		\item either the action of $G$ admits a unique fixed point and $G_+$ is abelian, or
		\item the actions of $G$ is semi-conjugate to an action by affine transformations. 
	\end{itemize}
\end{theorem}

Proofs of Solodov's theorem  \cite[Theorem 3.21]{Solodov} were also given by Barbot \cite[Théorème 2.8]{Barbot} and Kova\v cevi\'c \cite{Ko1}, see also the discussion in Ghys \cite[Theorem 6.12]{ghys-circle} and in Farb--Franks \cite{Farb-Franks}.

The main result of this note is about the next case which is not covered by the previous theorems, that is, when $N=2$.

\begin{theoremA}\label{t.2R}Consider a subgroup $G\leq \homeo(\mathbb{R})$ with at most $2$ fixed points. Then we have two possibilities:
	\begin{itemize}
		\item $G_+$ is abelian, and every fixed point of a non-trivial element of $G_+$ is a \mbox{global fixed point of $G_+$,} or
		\item the action of $G$ is semi-conjugate to an action by affine transformations. 
	\end{itemize}
\end{theoremA}

In other words, there is no new interesting group with at most two fixed points.
In fact, the group of affine transformations $\aff(\RR)$ is the only known example of group with at most $N$ fixed points, for any $N\ge 1$, which acts minimally on the real line. This suggests the following conjecture.

\begin{conjecture}\label{c.NR}
	Let $N\in \NN$ and let $G\leq \homeo(\mathbb{R})$ be  a subgroup  with at most $N$ fixed points. Let $I\subset \RR$ be a maximal interval without global fixed points for $G_+$.  Then the restriction of the action of $G$ to $I$ is semi-conjugate to the action by affine transformations.
\end{conjecture}

\begin{remark}
	The above conjecture holds true under higher regularity assumptions. Indeed, a result by Akhmedov \cite{akhmedov1,MR3519414} gives that any subgroup of $\diffeo_+^{r}([0,1])$, with $r>1$, with at most $N$ fixed points is solvable. 	
	On the other hand, a classical result of Plante \cite{Plante} gives that if a solvable group acts on the line such that the set of fixed point of any element acting non-trivial is discrete, then the action is semi-conjugate to an affine action. Using results from Bonatti, Monteverde, Navas and Rivas \cite{BMNR}, it is not difficult to conclude that if the affine action is non-abelian, then the semi-conjugacy is actually a conjugacy.
\end{remark}

\begin{small}
\subsection*{Acknowledgments.}
	We are grateful to Christian Bonatti and Michele Triestino for the very useful
	conversations and to Lorenzo D\'iaz, Andr\'es Navas and Maxime Wolff for their valuable remarks on the manuscript. This work was originally the content of the author's master thesis at PUC-Rio, and then appeared in the appendix of the author's PhD thesis.
	
	The author has been partially supported by the project MATH AMSUD DGT -- Dynamical Group Theory (22-MATH-03), the  project ANR Gromeov (ANR-19-CE40-0007). Our host department IMB receives support from the EIPHI Graduate School (ANR-17-EURE-0002).
\end{small}

\section{Preliminaries}

Let us first introduce the classical notion of (semi-)conjugacy for group actions on the line.
We say that a function $h:\RR\to\RR$ is a \emph{monotone proper map} if $h$ is non-decreasing and $\lim_{x\to \pm\infty}h(x)=\pm \infty$.
Let $G$ and $F$ be two subgroups of $\homeo(\RR)$. We say that the action of $G$ is \emph{(positively) semi-conjugate} to $F$ if there exist a surjective morphism $\theta: G\rightarrow F$, and a continuous, monotone proper map $h:\RR\rightarrow\RR$, which is equivariant, in other words 
$$\theta(g)h=hg,\quad \text{ for every } g\in G.$$
The map $h$ is called a \emph{(positive) semi-conjugacy}, and when $h$ is a homeomorphism, we say that it is a \emph{(positive) conjugacy}, in which case we also say that $G$ and $F$ are \emph{(positively) conjugate}.

There exists a natural equivalence relation for subgroups of $\homeo(\RR)$ given by $$G\sim F\quad \text{ if, and only if, } \quad G \text{ and } F \text{ are conjugate}.$$
However, semi-conjugacy does not define directly an equivalence relation (but one can slightly modify the definition so that it actually defines an equivalence relation, see for instance \cite{Kim-Koberda-Mj}).

The notion of conjugacy and semi-conjugacy can be transferred to homeomorphisms by stating that $f$ and $g$ in $\homeo(\RR)$ are (semi-)conjugate by the (semi-)conjugacy $h$ if the subgroups $\langle f\rangle$ and $\langle g \rangle$ are (semi-)conjugate by the (semi-)conjugacy $h$.
In this case, the surjective morphism $\theta:\langle f\rangle\rightarrow \langle g \rangle$ is the identity or the inversion, which implies that $fh=hg$ or $f^{-1}h=hg$. If we are in the second case, we replace our notation by saying that $f^{-1}$ and $g$ are (semi-)conjugate by the (semi-)conjugacy $h$.

Given a group action on a space $X$, we define the \emph{minimal invariant subset} as the smallest (by inclusion) closed non-empty subset of $X$ such that the orbit of every point is contained in the subset. Such subsets always exist for actions on compact topological spaces. This is however not the case for general actions on non-compact spaces. We have however the following result.

\begin{lemma}\label{l.existencia_minimal}
	Let $G$ be a subgroup of $\homeo(\RR)$ with at most $N$ fixed points, then $\RR$ contains a minimal invariant subset for the action of $G$.
	
	Moreover, the closure of any $G$-invariant subset contains a minimal invariant subset.
\end{lemma}

The proof of this result will be a direct application of the following stronger lemma (see for instance \cite[Proposition 2.1.12]{navas-book}).

\begin{lemma}\label{l.existencia_minimal_intervalo}
	Let $G$ be a subgroup of $\homeo(\RR)$ and $I$ a bounded open interval of $\RR$ such that the orbit of any point $x\in \RR$ intersects $I$, then the closure of any orbit $\overline{G.x}$ contains a minimal invariant subset for the action of $G$.
\end{lemma}
%
%

\begin{proof}[Proof of Lemma \ref{l.existencia_minimal}]
	We will assume that there are no global finite orbits for the action of $G$, otherwise every orbit would either be a minimal invariant subset or contains one in its closure. Therefore, there are no global fixed points to the action of $G_+$ and,	
	since every element has at most $N$ fixed points, we can choose $N+1$ elements $\{g_1,\ldots,g_{_{N+1}}\}\in G_+$ such that no point is fixed by all $g_n$ and we define the function $F:\RR\rightarrow \RR$ as:
	$$F(x):=\max_{n\in\{1,\ldots,N+1\}}{g_n^{\pm 1}(x)}.$$
	One can observe that $F$ is increasing, continuous and $F(x)>x$, for all $x\in\RR$. Now, let $I$ be any bounded open interval containing $[0, F(0)]$ and we claim that every orbit intersects $I$. Indeed, for all $x\in\RR$ there exists $k\in\ZZ$ such that $F^k(x)< 0$ and $F^{k+1}(x)\geq 0$, so we have $$F^{k+1}(x)=F(F^k(x))\in [0,F(0))\subset I.$$
	Therefore, the orbit of every point intersects the interval $I$ and, by Lemma \ref{l.existencia_minimal_intervalo}, we conclude that the closure of the orbit of any point contains a minimal invariant subset for the action of $G$.
\end{proof}

\begin{lemma}
	Let $G$ be an abelian subgroup of $\homeo(\RR)$ with at most $N$ fixed points. If there exists an element $f\in G$ with $\fix(f)\neq\emptyset$, then the minimal invariant subset of the action of $G$ is finite. Moreover, there exists a minimal invariant subset contained in $\fix(f)$, for every $f\in G$.
\end{lemma}
\begin{proof}
	Let $x\in X$ be a fixed point of $f\in G$, then for every $g\in G$ we have
	$$fg(x)=gf(x)=g(x).$$
	Therefore, $G\cdot x\subset\fix(f)$ for every $x\in\fix(f)$, which follows that $\fix(f)$ is a closed and invariant subset of $X$. Now, by Lemmas \ref{l.existencia_minimal} and \ref{l.existencia_minimal_intervalo}, the closure of the orbit $\overline{G.x}$ contains a minimal invariant subset for the action of $G$ and since it is contained in $\fix(f)$, we conclude that there exists a minimal invariant subset contained in $\fix(f)$.
\end{proof}
\begin{corollary}\label{c.abelian_global_fixed}
	Let $G$ be an abelian subgroup of $\homeo(\RR)$ with at most $N$ fixed points, then every fixed point is globally fixed.
\end{corollary}

This corollary comes from the fact that the only finite minimal invariant subsets on $\RR$ for an abelian group are single points which are globally fixed. The complete classification for the minimal invariant subsets of group actions on $\RR$ is given in the following classical theorem (the proof in \cite{ghys-circle} for the case of actions on the circle can be easily adapted to the case of actions on the line).

\begin{theorem}\label{t.minimal_R}
	Let $G$ be a subgroup of $\homeo(\RR)$ with a minimal invariant subset $M$. Then, the closure of any orbit contains a minimal invariant subset for the action of $G$ and moreover there are four mutually exclusive possibilities.
	\begin{itemize}
		\item[1.] All the orbits are dense and $M=\RR$ is the only minimal invariant subset.
		\item[2.] All minimal invariant subsets are global finite orbits.
		\item[3.] All minimal invariant subsets are  closed orbits of $G$, which are discrete and unbounded. Moreover, there exists an element $g\in G$ without fixed points such that all minimal invariant subsets are orbits of $g$.
		\item[4.] There exists a unique minimal invariant subset $M$ which is an unbounded Cantor set, i.e.\ perfect, totally disconnected and unbounded.
	\end{itemize}
\end{theorem}

\begin{remark}\label{r.closed_orbit}
	Note that in the case of closed orbits as minimal invariant subset, the action is semi-conjugate to an action by isometries, in particular $G_+$ is semi-conjugate to a group of translations.
\end{remark}

\section{Group actions on the line with at most 2 fixed points}\label{s.2fixedpoints}
The main purpose of this section is to prove Theorem \ref{t.2R}. For such we first analyses the situation when the action admits a global fixed point, and this is essentially based on Solodov's theorem (Theorem \ref{t.solodov}).

\begin{lemma}\label{l.2_global_fixed_points}
	Let $G$ be a subgroup of $\homeo_+(\RR)$ with at most $2$ fixed point, admitting a global fixed point. Then $G$ is abelian and any point which is fixed by a non-trivial element is globally fixed. 
\end{lemma}
\begin{proof}
	Let $x$ be the global fixed point and denote by $G_{-\infty}\le \homeo_+((-\infty,x))$ and $G_{+\infty}\le \homeo_+((x,+\infty))$ the subgroups obtained by considering the restriction of the action of $G$ to the two $G$-invariant half-lines, respectively. Note that the morphisms $G\to G_{\pm\infty}$ are both isomorphisms, as elements in one of the kernels fix a half-line and thus are globally trivial.
	Notice that if either $G_{-\infty}$ or $G_{+\infty}$ is abelian, this implies that $G$ is also abelian and so there is nothing to prove. So we will assume that none of them is abelian, and look for a contradiction. 
	Since every element of $G$ has at most one fixed point other than $x$, it follows that both $G_{-\infty}$ and $G_{+\infty}$ act with at most $1$ fixed point and therefore, by Solodov's theorem (Theorem \ref{t.solodov}), both $G_{-\infty}$ and $G_{+\infty}$ are semi-conjugate to non-abelian subgroups of $\aff_+(\RR)$. Now, any element $g\in G$ which is non-trivial in the abelianization $G/[G,G]$, gives in $G_{-\infty}$ and $G_{+\infty}$ elements which are semi-conjugate to homotheties, and thus $g$ admits at least three fixed points, giving the desired contradiction. 
\end{proof}

We next move to the case where the action has no global fixed points. We first introduce some terminology which will provide the combinatorial set-up for the core of the proof of Theorem \ref{t.2R}: for orientation-preserving homeomorphisms it is natural to consider whenever its graph is above or below the identity, and since we are restricting ourselves to homeomorphisms with at most $N$ fixed points, this information can be encoded in a finite string as in the following definition.

\begin{definition}\label{d.sign_homeo_r}
	Let $h\in\homeo_+(\RR)$ be an orientation-preserving homeomorphism with $\fix(h)=\{x_1, \ldots,x_n\}$, where $x_1<\cdots<x_n$. We say that \emph{$h$ is of type $(\varepsilon_0, ,\ldots,\varepsilon_n)$} for $\varepsilon_k\in\{+,-\}$ if the sign of $h-\idid$ restricted to $(x_k,x_{k+1})$ is $\varepsilon_k$ for every $k\in\{0,\ldots,n\}$ (here we write $x_0=-\infty$ and $x_{n+1}=+\infty$).
\end{definition}

\begin{remark}
	The notion presented in Definition \ref{d.sign_homeo_r} is invariant by conjugation and for every $h\in\homeo_+(\RR)$ and $n\in\NN^*$, $h$ and $h^n$ are of the same type, but $h$ and $h^{-1}$ are of opposite types (every coordinate has the opposite sign).
\end{remark}

We can now discuss the main technical lemma, which describes the case when the group contains an element with two fixed points.

\begin{lemma}\label{l.2_fixed_non_minimal}
	Let $G$ be a subgroup of $\homeo(\RR)$ with at most $2$ fixed points, whose action admits no global fixed point. Assume there exists an element $g\in G$ such that $\fix(g)=\{x,y\}$, with $x<y$. Then, the orbit of $x$ does not intersect the interval $(x,y)$. 
\end{lemma}

\begin{proof}
	First, observe that since $g$ has two distinct fixed points it cannot be an orientation-reversing homeomorphism, so $g\in G_+$. Furthermore, we will suppose that $G_+$ has no global fixed point, otherwise, by Lemma~\ref{l.2_global_fixed_points}, both $x$ and $y$ would be globally fixed by $G_+$ and the orbit of $x$ would be contained in $\{x,y\}$ which satisfies the statement.	
	Now, we will structure the proof by considering all the possible types for $g$. To start with, we argue that it is enough to restrict to the cases where the type of $g$ is $(+,+,+), (-,+,+), (+,+,-)$ and $(+,-,+)$, since all the $2^3$ possible types are reduced to these cases by considering $g^{-1}$.
	
	Arguing by contradiction, let us suppose that there exists $f\in G$ with $f(x)\in (x,y)$. We claim that we can take $f$ order preserving. Indeed, for $\varphi=fgf^{-1}\in G_+$ we have that $\varphi$ fixes the point $f(x)$ in the interior of $(x,y)$, thus either $\varphi(x)\in(x,y)$ or $\varphi^{-1}(x)\in (x,y)$ or $\varphi(x)=x$. The latter case implies that $f^{-1}(x)=y>x$ which is a contradiction then, by replacing $f$ for $\varphi$ or $\varphi^{-1}$, we conclude our claim.
	Now, let $f\in G_+$ with $f(x)\in(x,y)$. In this case $f(y)\neq y$, otherwise the subgroup defined by $\langle g, f \rangle\leq G_+$ would act with at most $2$ fixed points with one of them being globally fixed, and by Lemma \ref{l.2_global_fixed_points}, this would imply that $x$ is also globally fixed by $\langle f, g \rangle$, which is not the case.
	
	\smallskip 
	
	\noindent\textbf{Case 1.(a) $g$ is of type $(+,+,+)$ and $f(y)<y$.}
	Consider $h=fgf^{-1}\in G_+$, it is of type $(+,+,+)$ with $\fix(h)=\{f(x),f(y)\}$ satisfying that $x<f(x)<f(y)<y$, then one can observe that $h(x)\in(x,y)$ but $h(y)>y$, so we take $f'=h$ and we are reduced to the next case 1.(b).
	
	\smallskip 
	
	\noindent\textbf{Case 1.(b) $g$ is of type $(+,+,+)$ and $y<f(y)$.}	
	Consider $h=fgf^{-1}\in G_+$, it is of type $(+,+,+)$ with $\fix(h)=\{f(x),f(y)\}$ satisfying that $x<f(x)<y<f(y)$, therefore:
	\begin{itemize}
		\item $hg^{-1}(x)=h(x)>x$
		\item $hg^{-1}(f(x))<h(f(x))=f(x)$
		\item $hg^{-1}(y)=h(y)>y$
		\item $hg^{-1}(f(y))<h(f(y))=f(y)$
	\end{itemize}
	which implies that $hg^{-1}$ has a fixed point in each of the intervals $(x,f(x))$, $(f(x),y)$ and $(y,f(y))$. 
	This contradicts the assumption that $G$ has at most 2 fixed points.

	\smallskip 
	
	\noindent\textbf{Case 2.(a) $g$ is of type $(-,+,+)$ and $f(y)<y$.} 
	Consider $h=fgf^{-1}\in G_+$, it is of type $(-,+,+)$ with $\fix(h)=\{f(x),f(y)\}$ satisfying that $x<f(x)<f(y)<y$. Take $n\in\NN$ sufficiently large such that $g^{n}(x-1)<h(x-1)$ and $g^n(y+1)>h(y+1)$ then we have: 
	\begin{itemize}
		\item $hg^{-n}(h(x-1))>h(x-1)$
		\item $hg^{-n}(x)=h(x)<x$
		\item $hg^{-n}(y)=h(y)>y$
		\item $hg^{-n}(h(y+1))<h(y+1)$
	\end{itemize}
	which implies that $hg^{-n}$ has a fixed point in each of the intervals $(h(x-1),x)$, $(x,y)$ and $(y,h(y+1))$, which gives the desired contradiction.

	\smallskip 
	
	\noindent\textbf{Case 2.(b) $g$ is of type $(-,+,+)$ and $y<f(y)$.}
	Consider $h=fgf^{-1}\in G_+$, it is of type $(-,+,+)$ with $\fix(h)=\{f(x),f(y)\}$ satisfying that $x<f(x)<y<f(y)$. Then one can observe that $h^{-1}(x)\in(x,y)$ but $h^{-1}(y)<y$, so we take $f'=h^{-1}$ and we are back to the case 2.(a).
	
	\smallskip 
	
	\noindent\textbf{Case 3.(a) $g$ is of type $(+,+,-)$ and $f(y)<y$.} 
	Consider $h=fgf^{-1}\in G_+$, it is of type $(+,+,-)$ with $\fix(h)=\{f(x),f(y)\}$ satisfying that $x<f(x)<f(y)<y$. Take $n\in\NN$ sufficiently large such that $g^{-n}(x-1)<h^{-1}(x-1)$ and $g^{-n}(y+1)>h^{-1}(y+1)$ then we have:
	\begin{itemize}
		\item $h^{-1}g^{n}(h^{-1}(x-1))>h^{-1}(x-1)$
		\item $h^{-1}g^{n}(x)=h^{-1}(x)<x$
		\item $h^{-1}g^{n}(y)=h^{-1}(y)>y$
		\item $h^{-1}g^{n}(h^{-1}(y+1))<h^{-1}(y+1)$
	\end{itemize}
	which implies that $h^{-1}g^n$ has a fixed point in each of the intervals $(h^{-1}(x-1),x)$, $(x,y)$ and \mbox{$(y,h^{-1}(y+1))$,} which again gives the desired contradiction.

	\smallskip 
	
	\noindent\textbf{Case 3.(b) $g$ is of type $(+,+,-)$ and $y<f(y)$.}
	Consider $h=fgf^{-1}\in G_+$, it is of type $(+,+,-)$ with $\fix(h)=\{f(x),f(y)\}$ satisfying that $x<f(x)<y<f(y)$. Then one can observe that $g(x)\in(f(x),f(y))$ but $g(f(y))<f(y)$, so we take $g'=h$ and $f'=g$ and we are back to the case 3.(a).

	\smallskip 
	
	\noindent\textbf{Case 4.(a) $g$ is of type $(+,-,+)$ and $f(y)<y$.} This is by far the most complex case, and its proof will be given by a long construction.
	
	First, we remark that as the action of $G_+$ has no global fixed point, we can find two elements $u,v\in G_+$ such that $u(y)<x$ and $v(x)>y$.
	
	Define the following elements of $G_+$:
	\begin{itemize}
		\item $h=f g f^{-1}$
		of type $(+,-,+)$, with  $\fix(h)=\{f(x),f(y)\}=\{\tilde{x},\tilde{y}\}$
		\item $f_u=u g  u^{-1}$ of type $(+,-,+)$ with $\fix(f_u)=\{u(x),u(y)\}$
		\item $f_v=v g  v^{-1}$ of type $(+,-,+)$ with $\fix(f_v)=\{v(x),v(y)\}$
	\end{itemize}
	These points satisfy the order relation \[u(x)<u(y)<x<\tilde{x}<\tilde{y}<y<v(x)<v(y).\] 
	Now, observe that there exists $n_1\in \mathbb{N}$, such that for every $n\geq n_1$ one has $h^{-n}(x)<f_v^{-1}(x)$, and thus 
	\[h^n f_v h^{-n}(x)<h^n f_v(f_v^{-1}(x))=h^n(x)<\tilde{x}.\] 
	On the other hand, there exists $n_2\in \mathbb{N}$ such that for every $n\geq n_2$ one has $h^{n}f_v(\tilde{y})>y+2$,  and thus 
	\[h^n f_v h^{-n}(y)>h^n f_v h^{-n}(\tilde{y})=h^n f_v(\tilde{y})>y+2.\]	
	Let $n=\max\{n_1,n_2\}$ and define $s=h^n f_v h^{-n}\in G_+$, so we have $s(x)<\tilde{x}$ and $y+2<s(y)$.
	
	Similarly, one can observe that there exists $m_1\in \mathbb{N}$, such that for every $m\geq m_1$ one has $g^{-m}(f_u(x))>\tilde{y}$, and thus 
	\[g^{-m} f_u g^{m}(x)=g^{-m} f_u(x)>\tilde{y}.\] 
	On the other hand, there exists $m_2\in \mathbb{N}$ such that for every $m\geq m_2$ one has $g^{-m}(f_u(y))<y+1$,  and thus 
	\[g^{-m} f_u g^{m}(y)=g^{-m} f_u(y)<y+1.\]	
	Let $m=\max\{m_1,m_2\}$ and define $t=g^{-m} f_u g^{m}\in G$, so we have $\tilde{y}<t(x)$ and $t(y)<y+1$.
	
	We also observe that $\fix(s)=\{h^{n}(v(x)),h^{n}(v(y))\}$ and $\fix(t)=\{g^{-m}(u(x)),g^{-m}(u(y))\}$, and we have \[g^{-m}(u(x))<g^{-m}(u(y))<u(y)<x<y<v(x)<h^{n}(v(x))<h^{n}(v(y)).\]
	Now, let $z_1:=g^{-m}(u(x))\in\RR$, so that $z_1$ is a fixed point of $t$, which is smaller than both fixed points of $s$, which is of type $(+,-,+)$. Therefore, $t(z_1)=z_1$ and $s(z_1)>z_1$.
	Similarly, let $z_2:=h^{n}(v(y))\in\RR$, so that $z_2$ is a fixed point of $s$, which is larger than both fixed points of $t$, which is of type $(+,-,+)$. Therefore, $s(z_2)=z_2$ and $t(z_2)>z_2$.
	This leads to the following inequalities:
	\begin{itemize}
		\item $t(z_1)=z_1<s(z_1)$
		\item $t(x)>\tilde{y}>\tilde{x}>s(x)$
		\item $t(y)<y+1<y+2<s(y)$
		\item $t(z_2)>z_2=s(z_2)$
	\end{itemize}
	which implies that $ts^{-1}\in G$ is a non-trivial element with at least one fixed point in each of the intervals $(z_1,x)$, $(x,y)$ and $(y,z_2)$, adding up to at least $3$ fixed points.
	This contradicts the hypothesis that $G$ has at most $2$ fixed points. 
	
	\smallskip 
	
	\noindent\textbf{Case 4.(b) $g$ is of type $(+,-,+)$ and $y<f(y)$.}
	Consider $h=fgf^{-1}\in G_+$, it is of type $(+,-,+)$ with $\fix(h)=\{f(x),f(y)\}$ satisfying that $x<f(x)<y<f(y)$. Then one can observe that $h(x)\in(x,y)$ but $h(y)<y$, so we take $f'=h$ and we are reduced to the previous case 4.(a).

	As we have covered all possible situations,  we conclude that if there exists an element $g\in G$ such that $\fix(g)=\{x,y\}$, then the orbit of $x$ by $G$ does not intersect the interval $(x,y)$.
\end{proof}

\begin{proof}[Proof of Theorem \ref{t.2R}]
	If the action of $G_+$ admits a global fixed point, then Lemma \ref{l.2_global_fixed_points} implies that $G_+$ is abelian. 
	We assume next that the action of $G_+$ has no global fixed point. If $G$ has at most 1 fixed point, then we conclude by Solodov's theorem (Theorem \ref{t.solodov}) that the action is semi-conjugate to an affine action. 
	When $G$ contains elements with exactly two fixed points, Lemma \ref{l.2_fixed_non_minimal} implies that any interval of the form $(x,y)$, where $x,y$ are fixed by some non-trivial element, is wandering. Now, the complement $\RR\smallsetminus G.(x,y)$ contains a minimal invariant subset which is either discrete or uncountable. From Remark~\ref{r.closed_orbit}, if the minimal is discrete then $G_+$ is semi-conjugate to a group of translations, thus it is abelian and free of fixed points. For the latter case, a non-decreasing map $h:\RR\to\RR$ that collapses all these intervals to single points defines a semi-conjugacy of the action of $G$ to an action with at most 1 fixed point, so that we are reduced to the previous case.
\end{proof}

\bibliographystyle{plain}
\bibliography{biblio}

\smallskip

\noindent \textit{João Carnevale \\
Institut de Mathématiques de Bourgogne (IMB, UMR, CNRS 5584) \\
Université de Bourgogne \\
9 av.\ Alain Savary, 21000 Dijon, France\\}
\texttt{joao.carnevale@u-bourgogne.fr}

 \end{document}